\documentclass[12pt]{amsart}

\usepackage[UKenglish]{babel}
\usepackage{lmodern}
\usepackage[margin=25mm,bottom=35mm,footskip=15mm,top=35mm]{geometry}

\usepackage{
  enumerate,
  url,
  amssymb,
  amsmath,
  xstring,
  enumitem,
  hyperref,
  subcaption,
  floatrow,
  setspace
}

\usepackage[numbers]{natbib}
\usepackage{xcolor}
\usepackage{tikz}
\usetikzlibrary{calc,arrows,shapes}

\usepackage{orcidlink}
\usepackage{listings}
\usepackage{mathtools}

\tikzset{
  basepoly/.style={draw=black!35, line width=.6pt},
  guide/.style={draw=black!25, dashed, line width=.5pt},
  hole/.style={draw=blue!70!black, line width=1.3pt},
  chain/.style={draw=black!65, line width=.8pt},
  vertex/.style={circle, fill=black, inner sep=1.7pt},
  chosen/.style={circle, fill=blue!70!black, inner sep=2.15pt},
  added/.style={circle, fill=violet!80!black, inner sep=2.0pt},
  omitted/.style={circle, draw=red!70!black, fill=white, line width=.7pt, inner sep=2.0pt},
  blockpt/.style={circle, fill=black, inner sep=1.45pt},
  blockchosen/.style={circle, fill=blue!70!black, inner sep=2.0pt}
}

\usepackage{aliascnt}
\usepackage[capitalise,noabbrev]{cleveref}

\crefname{appsec}{Appendix}{Appendices}
\crefformat{equation}{#2(#1)#3}

\theoremstyle{plain}

\newtheorem{theorem}{Theorem}[section]

\newaliascnt{proposition}{theorem}
\newtheorem{proposition}[proposition]{Proposition}
\aliascntresetthe{proposition}

\newaliascnt{lemma}{theorem}
\newtheorem{lemma}[lemma]{Lemma}
\aliascntresetthe{lemma}

\newaliascnt{claim}{theorem}

\aliascntresetthe{claim}

\newaliascnt{corollary}{theorem}

\aliascntresetthe{corollary}

\newaliascnt{conjecture}{theorem}

\aliascntresetthe{conjecture}

\newaliascnt{observation}{theorem}

\aliascntresetthe{observation}

\newtheorem*{question*}{Question}

\theoremstyle{definition}

\newaliascnt{definition}{theorem}

\aliascntresetthe{definition}

\newaliascnt{question}{theorem}

\aliascntresetthe{question}

\newaliascnt{example}{theorem}

\aliascntresetthe{example}

\theoremstyle{remark}
\newtheorem*{remark}{Remark}

\crefname{theorem}{Theorem}{Theorems}
\crefname{proposition}{Proposition}{Propositions}
\crefname{lemma}{Lemma}{Lemmas}
\crefname{claim}{Claim}{Claims}
\crefname{corollary}{Corollary}{Corollaries}
\crefname{conjecture}{Conjecture}{Conjectures}
\crefname{observation}{Observation}{Observations}
\crefname{definition}{Definition}{Definitions}
\crefname{question}{Question}{Questions}
\crefname{example}{Example}{Examples}

\Crefname{theorem}{Theorem}{Theorems}
\Crefname{proposition}{Proposition}{Propositions}
\Crefname{lemma}{Lemma}{Lemmas}
\Crefname{claim}{Claim}{Claims}
\Crefname{corollary}{Corollary}{Corollaries}
\Crefname{conjecture}{Conjecture}{Conjectures}
\Crefname{observation}{Observation}{Observations}
\Crefname{definition}{Definition}{Definitions}
\Crefname{question}{Question}{Questions}
\Crefname{example}{Example}{Examples}

\newcommand{\xqed}[1]{%
  \leavevmode\unskip\penalty9999 \hbox{}\nobreak\hfill
  \quad\hbox{\ensuremath{#1}}%
}
\newcommand{\Endofdef}{\xqed{\lozenge}}

\title[Many holes but no large one]{Many holes but no large one: maximizing \(k\)-holes while forbidding \((k+1)\)-holes}

\title[Many holes but no large one]{
Many holes but no large one: maximizing \(k\)-holes while forbidding \((k+1)\)-holes
}

\author[M. Andri\v{c}ik, A. Dom\'anyov\'a, A. D\v{z}avoronok, A. D\v{z}uklevski, and M. \v{S}afr\'anek]{
Martin Andri\v{c}ik
\and
Alica Dom\'anyov\'a
\and
Adam D\v{z}avoronok
\and
Aleksa D\v{z}uklevski\
\and
Matou\v{s} \v{S}afr\'anek
}

\address{
Department of Applied Mathematics,
Faculty of Mathematics and Physics,
Charles University,
Czech Republic
}

\email{martin.andricik@gmail.com}
\email{alica@domany.sk}
\email{adam.dzavoronok@mff.cuni.cz}
\email{aleksa@mff.kam.cuni.cz}
\email{matous.safranek@gmail.com}

\begin{document}

	\maketitle
 	
\begin{abstract}
We study the maximal number $m_{k,\ell,n}$ of empty convex $k$-gons
(\emph{$k$-holes}) determined by an $n$-point set in the plane in general position
that contains no empty convex $\ell$-gon, focusing on the first nontrivial case
$\ell=k+1$. Our main result determines the exact value in the small-excess regime:
for $n=k+a$ with $a\le k/2-1$, we prove
$ m_{k,k+1,k+a}=2^a.$
We also describe the structure of the extremal configurations attaining equality. Beyond this exact range,
we provide upper and lower bounds in the proportional regime $n=\alpha k$ and in the
regime where $k$ is fixed and $n$ tends to infinity. In the latter regime we prove that
$m_{k,k+1,n}=\Omega_k(n^{\lfloor k/3\rfloor})$ and
$m_{k,k+1,n}=O_k(n^{\lfloor k/2\rfloor+1})$.
\end{abstract}

\section{Introduction}

The classical theorem of Erd\H{o}s and Szekeres \cite{ErdosSzekeres1935} states that
for every integer \(n\ge 3\), every sufficiently large set of points in the plane in
general position contains \(n\) points in convex position. Erd\H{o}s later asked for the empty analogue of this statement \cite{Erdos1978}: for
which integers \(k\) must every sufficiently large planar point set in general position
contain an empty convex \(k\)-gon? We call such polygons 
\(k\)-\emph{holes}. It is easy to see that every set of at least five points contains a
\(4\)-hole, and Harborth proved that every set of \(10\) points contains a \(5\)-hole
\cite{Harborth1978}. On the other hand, Horton constructed arbitrarily large point sets
with no \(k\)-hole for any \(k\ge 7\), showing that the empty version of the
Erd\H{o}s-Szekeres theorem fails in this range \cite{Horton}. The remaining case
\(k=6\) is much more interesting: Nicol\'as and Gerken independently proved that every
sufficiently large point set contains a \(6\)-hole
\cite{Nicolas2007,Gerken2008}. More recently, Heule and Scheucher proved that every set
of \(30\) points in general position contains a \(6\)-hole, matching Overmars construction
of a \(29\)-point set with no \(6\)-hole~\cite{Overmars2002,HeuleScheucher2024}.

A natural quantitative companion to these existence questions is to count holes. For a
fixed \(k\), let \(h_k(n)\) denote the minimum number of \(k\)-holes that every
\(n\)-point set in general position must determine. The systematic study of this problem
goes back at least to Katchalski and Meir~\cite{KatchalskiMeir1988}, B\'ar\'any and
F\"uredi~\cite{BaranyFuredi1987}, and Dehnhardt~\cite{Dehnhardt1987}. For small polygons,
the asymptotic behaviour is well understood: it is well known that the minimum number of empty triangles and empty
convex quadrilaterals both grow quadratically with the size of the point set (see for example \cite{BaranyValtr2004}). For larger holes, the problem becomes
substantially harder. In particular, the case \(k=5\) has received considerable attention:
Aichholzer et al. proved the superlinear lower bound
\(\Omega(n\log(n)^{4/5})\)~\cite{AichholzerEtAl2020}, and very recently, Astudillo-Marb\'an and Sol\'e-Pi improved this to \(\Omega(n^{20/11})\)~\cite{AstudilloMarbanSolePi2026}. The value \(h_6(n)\) is known to grow at least linearly in \(n\) (e.g. \cite{BaranyValtr2004}). For larger holes, the guaranteed number drops
sharply, since
\(h_k(n)=0\) for every \(k\ge 7\), by Horton's construction.

In this paper, we look at the opposite extremal direction. If one asks for the maximum
possible number of \(k\)-holes in an \(n\)-point set without any further restriction, then
the answer is simply \(\binom{n}{k}\), attained by placing all points in convex position.
Thus, the unrestricted maximisation problem is trivial. We obtain a nontrivial problem by forbidding larger empty polygons. In other words, we ask how many \(k\)-holes a point set can have if it  contains no  empty convex \(\ell\)-gon. The most natural first case is \(\ell=k+1\): we want many holes of size \(k\), but no hole just one vertex larger.

For integers \( 3\le k<\ell\leq n\), let \(\mathcal{X}_{\ell,n}\) denote the family of all \(n\)-point sets in the plane in general position that contain no \(\ell\)-hole. For
\(X\in\mathcal{X}_{\ell,n}\), let \(M_k(X)\) be the number of \(k\)-holes determined by
\(X\), and define
$$
m_{k,\ell,n}:=\max\{M_k(X):X\in\mathcal{X}_{\ell,n}\}.
$$

We study \(m_{k,k+1,n}\) in two main regimes. The first is the proportional regime
\(n=\alpha k\), where \(k\) grows and each \(k\)-hole uses a positive fraction of all
points. In the initial range \(n=k+a\) with \(a\le k/2-1\), we determine the exact value.
The proof is based on a simple but useful geometric union lemma: under suitable
intersection assumptions, the union of three holes is again a hole. This turns the absence
of \((k+1)\)-holes into a strong uniqueness statement for how a \(k\)-hole can differ
from a fixed one.

\begin{theorem}\label{thrm1}
Let \(n=k+a\) and assume \(2a+2\le k\). Among all \(n\)-point sets in general
position with no \((k+1)\)-hole, the maximum possible number of \(k\)-holes is
\(2^a\). Furthermore, there exists a point set with \(n=k+a\) points, \(2^a\)
holes of size \(k\), and no \((k+1)\)-hole.
\end{theorem}

The lower bound construction for \Cref{thrm1} is the following. We start with a convex \(k\)-gon and  place
one additional point very close to some \(a\) pairwise nonadjacent edges.
Beyond this exact range, we continue to investigate the proportional regime $n=\alpha k$ for $\alpha\ge3/2$. There, the
answer is no longer determined exactly. However, for the upper bound  we use same proof strategy. We combine
geometric statements about unions of holes with some auxiliary hypergraph counting argument. For the lower bound, we present
separate chain constructions that produce many \(k\)-holes by taking consecutive points from disjoint chains, under the assumption that $n$ is bounded by $k^2/2$. We defer the exact statements until Section~\ref{sec:proportional}. 

The second main regime is when \(k\) is fixed and \(n\to\infty\). This is perhaps the most natural asymptotic. 
We prove the following upper bound.

\begin{theorem}\label{thm:upper-bound-fixed-k}
    Let $k$ be an  integer greater than or equal to $6$. Then 
    $$m_{k,k+1,n}\leq 
    \begin{cases} \frac{2n-k/2}{k}\binom{n}{k/2} & \text{if } 2\mid k, \\
                 
 \frac{\lceil k/2\rceil}{k}\binom{n}{\lfloor k/2\rfloor+1}                                    & \text{if } 2\nmid k .    %
        \end{cases}$$
\end{theorem}

To outline the proof, we sample alternating vertices from each \(k\)-hole. For example if \(k=2r\), from each \(2r\)-hole, we sample two disjoint alternating empty \(r\)-gons. We then show that, once such an \(r\)-gon is fixed, number of possible extensions to a \(2r\)-hole, from which it could be sampled, is at most linear in $n$. By double counting, we obtain the desired bound.

We complement this upper bound with a construction giving many \(k\)-holes and no
\((k+1)\)-hole. The main ingredient is Horton set construction. We replace the vertices of a convex polygon with small, very flat Horton blocks. Choosing three suitable consecutive
points from each block produces many \(k\)-holes, while the Horton property prevents any
hole from using too many points in a single block.

\begin{theorem}\label{thm:lower-bound-fixed-k}
    Let $k$ be an integer greater than or equal to $6$ and divisible by $3$. Then
    $$m_{k,k+1,n}\ge\left(\frac{3(n-k)}{2k}\right)^{k/3}=\Omega_k(n^{k/3}).$$
\end{theorem}

\begin{remark}
 One can recover the proof of \Cref{thm:lower-bound-fixed-k} bounds $m_{k,k+1,n}=\Omega_k(n^{\lfloor\frac{k}{3}\rfloor})$ for $k$ not divisible  $3$ . We sketch how one can recover it.
\end{remark}
Lastly, we remark that during the final preparation of this paper, Suk and Zhou independently proved
similar bounds in the fixed \(k\), large \(n\) regime
\cite{SukZhou2026}. The overlap between the two projects was discussed with
the authors of \cite{SukZhou2026}. 
\section{Exact regime: $n<\frac{3}{2}k$}
\subsection*{Upper bound proof}
 We first prove the following auxiliary lemma about the intersection patterns of three distinct holes, which will be crucial for the proof of \Cref{thrm1} and later on. 
\begin{lemma}\label{lem:three-holes-union}
Let \(K,L,M\) be holes in a point set \(X\). Assume that

\begin{enumerate}
\item no point of \(K\cup L\cup M\) belongs to exactly one of the three holes \(K,L,M\), and
\item at least two points belong to all three holes.
\end{enumerate}
Then \(K\cup L\cup M\) is a hole.
\end{lemma}

\begin{proof}
Set \( U:=K\cup L\cup M \). We show that every triangle spanned by three points of \(U\) is empty. This implies that \(U\) is a hole. Take any three distinct points \(x,y,z\in U\). If \(x,y,z\) all lie in one of the holes \(K,L,M\), then \(\triangle xyz\) is empty, since each of \(K,L,M\) is a hole. So suppose \(x,y,z\) do not all lie in a single one of \(K,L,M\). Therefore, we may assume, after possible relabelling, that
$$
x\in K\cap L,\qquad
y\in L\cap M,\qquad
z\in M\cap K.
$$

Choose two distinct points \(p,q\in K\cap L\cap M\), which exist by assumption (2). We claim that at least one of \(p,q\) lies outside \(\triangle xyz\). Suppose for contradiction that both \(p\) and \(q\) lie in \(\triangle xyz\). Consider triangles \(\triangle xyp,\triangle yzp,\triangle zxp\). By general position, $q$ lies in one of them, say $\triangle xyp$. However, $x,y,p\in L$, which is a hole, so the triangle has to be empty. This gives  the desired contradiction. We may therefore assume that $p\notin\triangle xyz$.

Now
$$
x,y,p\in L,\qquad y,z,p\in M,\qquad z,x,p\in K.
$$
Therefore, the three triangles
\(
\triangle xyp, \triangle yzp,\triangle zxp
\)
are all empty. Since \(p\notin \triangle xyz\), the triangles
cover \(\triangle xyz\). Hence, if each of them is empty, then \(\triangle xyz\) is empty as well.

\end{proof}
 We now proceed with the upper bound argument for $m_{k,k+1,k+a}$. It follows immediately from the following proposition.
\begin{proposition}\label{unique-subsets}
     Under the assumptions of \Cref{thrm1}, if $K$ is a fixed $k$-hole and $A=X\setminus K$ is the set of remaining $a$ points, then for every $B \subseteq A$, there is at most one $k$-hole that contains all the points in $B$ and whose vertices all lie in $K\cup B$. 
\end{proposition}
Since, there are $2^a$ subsets of $A$, and each contributes at most one $k$-hole (note that $a \leq \frac{k}{2}-1$, so we have to use points from $K$ as well, i.e. the $k$-hole has to be of the type described in \Cref{unique-subsets}), so the conclusion follows. We now give the proof, which is just a direct application of \Cref{lem:three-holes-union}.

\begin{proof}
   For the sake of contradiction, suppose that there exists $B\subseteq A$ such that there are two distinct $k$-holes  $H_1$ and $H_2$ such that 
   $$
   H_1\cap A=B \qquad  H_2\cap A=B.
   $$
   Define $C_i:=H_i\cap K$ for $i\in\{1,2\}$ and $M:=C_1\cup C_2\subseteq K$. Our goal is to apply \Cref{lem:three-holes-union} for $H_1,H_2,M$. Observe that $|C_i|=k-|B|$ and moreover, 
  \begin{align*}
  |H_1\cap H_2\cap M|
    &=|C_1\cap C_2|\\
    &=|C_1|+|C_2|-|M|\\
    &\ge 2(k-|B|)-k
     = k-2|B|
     \ge 2.
  \end{align*}
  where we used the fact $|M|\leq k$, $|B|\le|A|=a$, and that $k\ge 2a+2$. Lastly, we verify that each point of $H_1\cup H_2\cup M$ belongs to at least two of the three sets $H_1, H_2, M$. Indeed, $M\subseteq H_1\cup H_2$, $C_i\subseteq M$ and $H_i\setminus C_i=B\subseteq H_1\cap H_2$. So by \Cref{lem:three-holes-union} we obtain that points in $H_1\cup H_2\cup M $ form a hole; however, by our assumption, the $k$-hole $H_1$ is strictly contained in $H_1\cup H_2\cup M$ as $H_1 \neq H_2$. Thus, $H_1\cup H_2\cup M$ is a hole of size at least $(k+1)$, which is the desired contradiction.
\end{proof}
\begin{remark}
The bound \(k\ge 2a+2\) in \Cref{unique-subsets} and \Cref{thrm1} is the best possible. Figure~\ref{fig:sharpness-a3}  depicts a configuration with $a=3$ and $k=7$ where the conclusion does not follow. The $a_i$ points represent the points of \(A\). More generally, the same construction works for arbitrary \(b\): one replaces the three-point chains on the left and on the right of the figure with convex chains of length \(b\), and the three-points of \(A\) with a convex chain of \(b\) points. 
\begin{figure}[htbp]
\centering

\begin{tikzpicture}[scale=.6]
  \definecolor{myblue}{RGB}{20,35,190}
  \definecolor{myteal}{RGB}{52,126,126}
  \definecolor{myred}{RGB}{210,45,35}

  \coordinate (a1) at (-4,0.0);
  \coordinate (a2) at ( 0.0,-0.7);
  \coordinate (a3) at ( 4,0.0);

  \coordinate (b1) at ( 5.2,5.3);
  \coordinate (b2) at ( 6.1,4.7);
  \coordinate (b3) at ( 6.6,3.4);

  \coordinate (c1) at (-5.2,5.3);
  \coordinate (c2) at (-6.1,4.7);
  \coordinate (c3) at (-6.6,3.4);

  \coordinate (e) at (0,2.8);

  \draw[myblue,line width=1.2pt]
    (a1)--(a2)--(a3)--(e)--(c1)--(c2)--(c3)--cycle;

  \draw[myteal,line width=1.2pt]
    (a1)--(a2)--(a3)--(b3)--(b2)--(b1)--(e)--cycle;

  \draw[myred,line width=1.2pt]
    (c3)--(c1)--(b1)--(b3)--(e)--cycle;

  \foreach \p in {a1,a2,a3,b1,b2,b3,c1,c2,c3,e}
    \fill[black] (\p) circle (2.1pt);

  \node[font=\fontsize{10}{10}\selectfont]
    at ($(a1)+(-0.55,-0.25)$) {$a_1$};
  \node[font=\fontsize{10}{10}\selectfont]
    at ($(a2)+(0,-0.55)$) {$a_2$};
  \node[font=\fontsize{10}{10}\selectfont]
    at ($(a3)+(0.55,-0.25)$) {$a_3$};

  \node[font=\fontsize{10}{10}\selectfont]
    at ($(b1)+(0.35,0.55)$) {$b_1$};
  \node[font=\fontsize{10}{10}\selectfont]
    at ($(b2)+(0.55,0.10)$) {$b_2$};
  \node[font=\fontsize{10}{10}\selectfont]
    at ($(b3)+(0.60,-0.10)$) {$b_3$};

  \node[font=\fontsize{10}{10}\selectfont]
    at ($(c1)+(-0.35,0.55)$) {$c_1$};
  \node[font=\fontsize{10}{10}\selectfont]
    at ($(c2)+(-0.55,0.10)$) {$c_2$};
  \node[font=\fontsize{10}{10}\selectfont]
    at ($(c3)+(-0.60,-0.10)$) {$c_3$};

  \node[font=\fontsize{10}{10}\selectfont]
    at ($(e)+(0,0.55)$) {$e$};

\end{tikzpicture}

\caption{Construction depicting sharpness of Proposition~\ref{unique-subsets} for $k\ge2a+2$.}
\label{fig:sharpness-a3}
\end{figure}
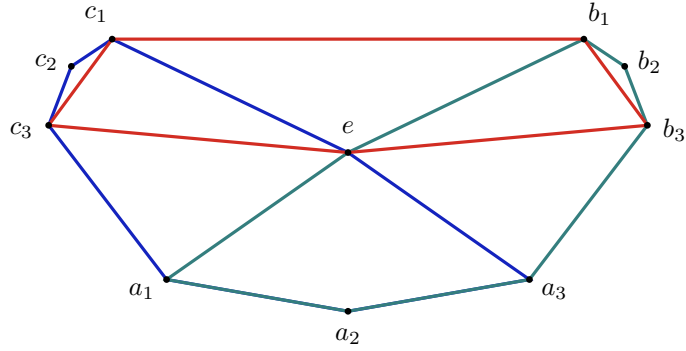
\end{remark}
\subsection*{Lower bound construction and structural description}
To finish the proof of \Cref{thrm1}, we proceed with the lower bound.

\begin{proposition}\label{prop:lower-bound}
If \(k\ge2a\), then
$
m_{k,k+1,k+a}\ge 2^a.
$
Equivalently, there exists a point set \(X\in \mathcal X_{k+1,k+a}\) with at least \(2^a\)
distinct \(k\)-holes.
\end{proposition}

\begin{proof} 
Let \(v_1,\dots,v_k\) be the vertices of a convex \(k\)-gon, and choose the pairwise nonadjacent edges \(e_i=v_{2i-1}v_{2i}\), for \(i\in[a]\). For each \(i\), place a point \(u_i\) just inside the polygon and sufficiently close to the midpoint of \(e_i\). Choose these points sufficiently close and generically so that every set obtained by taking \(u_i\) together with exactly one endpoint of \(e_i\) for each \(i\) and keeping all untouched vertices is in convex position, with every unchosen endpoint outside its convex hull. Thus, every such choice gives a \(k\)-hole, and the \(2^a\) choices give \(2^a\) distinct \(k\)-holes. 

We may also choose the points sufficiently close to their edges that \(u_i\in\operatorname{int}\triangle v_{2i-1}v_{2i}x\) for every \(x\in X\setminus\{v_{2i-1},u_i,v_{2i}\}\). Now let \(Y\subseteq X\) have more than \(k\) points. If \(Y\) contained at most two points from each triple \(\{v_{2i-1},u_i,v_{2i}\}\), then \(|Y|\le k-2a+2a=k\). Hence \(Y\) contains one of these triples and also some point \(x\) outside it. But then \(u_i\) lies inside \(\triangle v_{2i-1}v_{2i}x\subseteq\operatorname{conv}(Y)\), so \(Y\) is not in convex position. Therefore \(X\) contains no \((k+1)\)-hole. 
\end{proof}

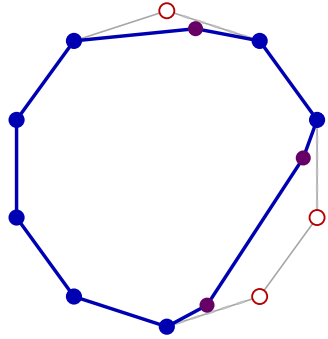
\begin{figure}[htbp]
\centering
\begin{tikzpicture}[scale=.95]
\def\R{2.2}
\def\nin{10}
\foreach \i in {1,...,10}{
  \pgfmathsetmacro\ang{90-360*(\i-1)/\nin}
  \coordinate (v\i) at ({\R*cos(\ang)},{\R*sin(\ang)});
}
\coordinate (u1) at (0.4,1.95);
\coordinate (u2) at (1.9,0.15);
\coordinate (u3) at (0.56,-1.90);

\draw[basepoly] (v1)--(v2)--(v3)--(v4)--(v5)--(v6)--(v7)--(v8)--(v9)--(v10)--cycle;
\draw[guide, line width=.9pt] (v1)--(v2) (v3)--(v4) (v5)--(v6);

\draw[hole] (u1)--(v2)--(v3)--(u2)--(u3)--(v6)--(v7)--(v8)--(v9)--(v10)--cycle;

\foreach \i in {2,3,6,7,8,9,10}{\node[chosen] at (v\i) {};}
\foreach \i in {1,4,5}{\node[omitted] at (v\i) {};}
\node[added] at (u1) {}; \node[added] at (u2) {}; \node[added] at (u3) {};

\end{tikzpicture}
\caption{The local lower-bound construction for $m_{k,k+1,k+a}\ge 2^a$.}
\label{fig:local-edge-replacement}
\end{figure}
We also record the following structural consequence of equality in \Cref{thrm1}. In any extremal configuration, once a
base \(k\)-hole \(K\) is fixed, each point outside \(K\) is paired with a distinct
point of \(K\), and every  \(k\)-hole is obtained by making swaps between the points.
\begin{proposition}\label{thm:structural-stability}
Let $X$ be a point set with $|X| = k+a$ without $(k+1)$-holes, where $k \ge 2a+2$ and $a \ge 2$. Assume $X$ contains exactly $2^a$ distinct $k$-holes. Let $K$ be a fixed $k$-hole and $A = X \setminus K$. Then there exists a set of $a$ distinct points $Y = \{y_x : x \in A\} \subseteq K$ such that the $2^a$ distinct $k$-holes of $X$ are exactly the sets $H_B$ for all subsets $B \subseteq A$, given by
$$
H_B = \left(K \setminus \{y_x : x \in B\}\right) \cup B.
$$
\end{proposition}


\begin{proof}

By \Cref{unique-subsets} and the maximality of the configuration, for each $B\subseteq A$ there is exactly one $k$-hole $H_B$ with $H_B\cap A=B$. Define $G_B:= K\setminus H_B$, the replacing points for $B$. So we wish to prove that $G_B = \{y_x:x\in B\}$. We divide the proof in two steps: 

Firstly, we observe that for $B'\subseteq B$ we have $G_{B'} \subseteq G_B$. Indeed, consider the three holes
$$
    K \setminus (G_B\cap G_{B'}) \subseteq K
$$
$$
    B'\cup(K \setminus G_B) \subseteq H_B 
$$
$$
    B'\cup(K \setminus G_{B'}) = H_{B'}  
$$
for which the conditions of the \cref{lem:three-holes-union} are readily verified. Indeed, points in $B'$ appear in two holes, $K\setminus (G_B\cup G_{B'})$ (which is by assumption $k \geq 2a + 2$ of size at least $k-2a = 2$) in three, and the points in $G_B\setminus G_{B'}, G_{B'}\setminus G_B$ in two of them. Their union $B'\cup (K \setminus (G_B\cap G_{B'}))$ is therefore a hole. This forces $|B'|+|K \setminus (G_B\cap G_{B'})|\leq k$ or $|G_B\cap G_{B'}|\geq |B'| = |G_{B'}|$ which infers the first conclusion.

Secondly, let $C\cup C'= B,\: C\cap C' = \emptyset$.  By the first step, $G_C\cup G_{C'}\subseteq G_B$, we show that $G_C\cup G_{C'}\subseteq G_B$. Consider the three holes 
$$
B\cup (K\setminus G_B) \subseteq H_B
$$
$$
C \cup (K\setminus (G_C \cup G_{C'})) \subseteq H_C
$$
$$
C' \cup (K\setminus (G_C \cup G_{C'})) \subseteq H_{C'}
$$
again clearly fulfilling the conditions of the \cref{lem:three-holes-union}. Here the points in $C\cup C' = B$ appear in two holes, $K \setminus G_B$ in three and $G_B\setminus(G_C \cup G_{C'})$ in two of them. Now their union $B \cup (K\setminus(G_C\cup G_{C'}))$ being a hole forces $|B|+|K\setminus(G_C\cup G_{C'})|\leq k$, then $|G_C\cup G_{C'}|\ge |B|=|G_C|+|G_{C'}|$. Hence $G_C\cap G_{C'}=\emptyset$. Since
$G_C\cup G_{C'}\subseteq G_B$ and
$|G_C\cup G_{C'}|=|B|=|G_B|$, we conclude that
$G_B=G_C\mathbin{\cup}G_{C'}$. Now the claim $G_B = \{y_x:x\in B\}$ is readily verified by induction on $|B|$, the ground case $|B|=1$ being trivial.
\end{proof}

\section{Proportional $n=\alpha k$ regime}\label{sec:proportional}
\subsection*{Upper bound proof}
For \(n\ge k+\lfloor k/2\rfloor\), we are going to prove the following upper bound.
\begin{theorem}\label{thm:three-halves}
Let \(n\ge k+\lfloor k/2\rfloor\). Then
$$
m_{k,k+1,n}
\le
\frac{5\cdot 3^{\lfloor k/2\rfloor}-3}{2}
\frac{\binom{n}{k}}{\binom{k+\lfloor k/2\rfloor}{k}} .
$$
\end{theorem}

We will again start with a proof of a lemma similar to \Cref{lem:three-holes-union}.
\begin{lemma}\label{lem:two-thirds-union}
Let \(\mathcal C=\{H_1,\dots,H_m\}\) be a collection of \(m\) distinct holes in a
point set \(X\). Let \(S\subseteq \bigcup_{i=1}^m H_i\) be such that every point of
\(S\) belongs to more than \(2m/3+1\) holes in \(\mathcal C\). Then \(S\) is a hole.
\end{lemma}
\begin{proof}
It suffices to show that every triangle spanned by three points of \(S\) is empty.
For each point among \(x,y,z\in S\), the number of holes in \(\mathcal C\) not containing it is at most $m/3-1$. Their union, which is the set of holes that contain at most two points of \(x,y,z\), is therefore of size at most $3(m/3-1)<m$. So there is a hole that contains \(x,y,z\) and therefore \(\triangle xyz\) is empty.

\end{proof}

Now, to show \Cref{thm:three-halves}, it suffices to establish the following proposition.

\begin{proposition}\label{prop:hypergraph-reduction}
 Let $n \ge k$, and let $F=(V,E)$ be a $k$-uniform hypergraph with $|V(F)|=n$. If $F$ satisfies that for every subhypergraph $H \subseteq F$ with $|E(H)| = m$, the number of vertices with a degree strictly greater than $2m/3 + 1$ is at most $k$, then for every integer $r$ such that $0 \le r \le n-k$, we have
$$|E({F})| \le \frac{5 \cdot 3^r - 3}{2}\frac{\binom{n}{k}}{\binom{k+r}{k}}.$$
\end{proposition}
\Cref{thm:three-halves} is then immediate for \(r=\lfloor k/2\rfloor\) when one considers the hyperedges in the hypergraph to be the \(k\)-holes in the point set and directly applies \Cref{lem:two-thirds-union}.
\begin{proof}
For $r\ge0$, set
$$
c_r:=\frac{5\cdot 3^r-3}{2}.
$$
We prove by induction on $r$ that every set $T\subseteq V(F)$
with $|T|=k+r$ spans at most $c_r$ edges.
For $r=0$, the set $T$ has size $k$, so it spans at most one edge, and
$c_0=1$. Now let $r\ge1$, and put $m:=|E(F[T])|$. By the assumption applied to
the induced subhypergraph $F[T]$, at most $k$ vertices of $T$ have degree
greater than $2m/3+1$. Since $|T|=k+r>k$, there is a vertex $v\in T$
with
$$
\deg_{F[T]}(v)\le \frac{2m}{3}+1.
$$
It follows that induced hypergraph
$$
|E(F[T\setminus{v}])|
=m-\deg_{F[T]}(v)\ge \frac{m}{3}-1.
$$
By the induction hypothesis, the left-hand side is at most
$c_{r-1}$. Hence $m\le 3c_{r-1}+3=c_r$, proving the claim. We now double-count pairs $(e,T)$ such that $e\in E(F)$,
$T\in\binom{V(F)}{k+r}$, and $e\subseteq T$. Every edge $e$ is
contained in exactly $\binom{n-k}{r}$ such sets $T$, while every $T$
contains at most $c_r$ edges. Therefore
$$
|E(F)|\binom{n-k}{r}
\le c_r\binom{n}{k+r}.
$$
Finally,
$
\frac{\binom{n}{k+r}}{\binom{n-k}{r}}
=\frac{\binom{n}{k}}{\binom{k+r}{k}},
$
implies that
$$
|E(F)|
\le
\frac{5\cdot 3^r-3}{2}
\frac{\binom{n}{k}}{\binom{k+r}{k}}.
$$
\end{proof}
\subsection*{Lower bound chain construction}
We finish the section with a lower-bound construction. The construction below is tailored to the regime
$n<k^{2}/2$, and yields a product-type family of $k$-holes while still avoiding $(k+1)$-holes.

\begin{proposition}\label{prop:chain_construction}
 Let $k$ be even, let $k+1 \leq n < k^{2}/2$, and assume that $t:=2n/k$ is an
integer. Then $$m_{k,k+1,n}\ge (t-1)^{k/2} = (2n/k-1)^{k/2}.$$
\end{proposition}

Without the technical assumption of $t$ being an integer, one can recover $m_{k,k+1,n}\ge (t-2)^{k/2} = (2n/k-2)^{k/2}$ merely by modifying the following construction.

\begin{proof}
Let \(v_1,\dots,v_k\) be the vertices of a regular \(k\)-gon in cyclic
order, and for \(i\in[k/2]\) let \(e_i=v_{2i-1}v_{2i}\). Replace each
\(e_i\) by a sufficiently flat, strictly concave chain
\(C_i=(p_{i,1},\dots,p_{i,t})\), where
\(p_{i,1}=v_{2i-1}\) and \(p_{i,t}=v_{2i}\).

We choose the chains sufficiently close to the original edges that the
following properties hold. First, the line through
any two consecutive points of \(C_i\) separates the remaining points of
\(C_i\) from all the other chains. Second, whenever \(a,b,c\in C_i\)
occur in this order and \(q\) lies on another chain, we have
\(b\in\operatorname{int}\triangle acq\). These properties hold simultaneously after taking the finitely many chains sufficiently flat. 

Put \(X:=C_1\cup\cdots\cup C_{k/2}\), so that
\(|X|=(k/2)t=n\). For every
\(\mathbf j=(j_1,\dots,j_{k/2})\in\{1,\dots,t-1\}^{k/2}\), let
$
H_{\mathbf j}:=
\bigcup_{i=1}^{k/2}\{p_{i,j_i},p_{i,j_i+1}\}.
$
By the first  property, each selected consecutive pair forms
an exposed edge of \(\operatorname{conv}(H_{\mathbf j})\), while every
omitted point of that chain lies outside this convex hull. Hence
\(H_{\mathbf j}\) is a \(k\)-hole. Different vectors \(\mathbf j\) produce different holes, so \(X\) contains at least \((t-1)^{k/2}\) distinct \(k\)-holes.

It remains to exclude \((k+1)\)-holes. Let \(H\) be a hole meeting at least two chains, and choose \(q\in H\setminus C_i\). If \(H\cap C_i\)
contained three points \(a,b,c\) in this order, then
 \(b\in\operatorname{int}\triangle acq\), contradicting the convexity of \(H\). Thus \(|H\cap C_i|\le2\).
Consequently, every hole meeting
more than one chain has at most two points from each of the \(k/2\) chains, and hence at most \(k\) vertices. A hole contained in one chain has at most \(t\leq k\) vertices because
\(n\leq k^2/2\). Therefore \(X\) contains no \((k+1)\)-hole~and
$$
m_{k,k+1,n}\ge M_k(X)\ge(t-1)^{k/2}
=(2n/k-1)^{k/2}.
$$
\end{proof}
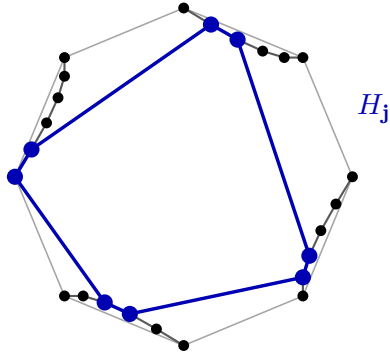
\begin{figure}[htbp]
\centering
\begin{tikzpicture}[scale=.95]
\coordinate (v1) at (0.00,2.35);
\coordinate (v2) at (1.66,1.66);
\coordinate (v3) at (2.35,0.00);
\coordinate (v4) at (1.66,-1.66);
\coordinate (v5) at (0.00,-2.35);
\coordinate (v6) at (-1.66,-1.66);
\coordinate (v7) at (-2.35,0.00);
\coordinate (v8) at (-1.66,1.66);
\draw[basepoly] (v1)--(v2)--(v3)--(v4)--(v5)--(v6)--(v7)--(v8)--cycle;

\coordinate (c1-0) at (0.00,2.35); \coordinate (c1-1) at (0.38,2.12); \coordinate (c1-2) at (0.75,1.91); \coordinate (c1-3) at (1.10,1.75); \coordinate (c1-4) at (1.40,1.66); \coordinate (c1-5) at (1.66,1.66);
\coordinate (c2-0) at (2.35,0.00); \coordinate (c2-1) at (2.12,-0.38); \coordinate (c2-2) at (1.91,-0.75); \coordinate (c2-3) at (1.75,-1.10); \coordinate (c2-4) at (1.66,-1.40); \coordinate (c2-5) at (1.66,-1.66);
\coordinate (c3-0) at (0.00,-2.35); \coordinate (c3-1) at (-0.38,-2.12); \coordinate (c3-2) at (-0.75,-1.91); \coordinate (c3-3) at (-1.10,-1.75); \coordinate (c3-4) at (-1.40,-1.66); \coordinate (c3-5) at (-1.66,-1.66);
\coordinate (c4-0) at (-2.35,0.00); \coordinate (c4-1) at (-2.12,0.38); \coordinate (c4-2) at (-1.91,0.75); \coordinate (c4-3) at (-1.75,1.10); \coordinate (c4-4) at (-1.66,1.40); \coordinate (c4-5) at (-1.66,1.66);

\foreach \i in {1,...,4}{
  \draw[chain] (c\i-0)--(c\i-1)--(c\i-2)--(c\i-3)--(c\i-4)--(c\i-5);
  \foreach \j in {0,...,5}{\node[blockpt] at (c\i-\j) {};}
}

\draw[hole] (c1-1)--(c1-2)--(c2-3)--(c2-4)--(c3-2)--(c3-3)--(c4-0)--(c4-1)--cycle;
\foreach \p in {c1-1,c1-2,c2-3,c2-4,c3-2,c3-3,c4-0,c4-1}{\node[chosen] at (\p) {};}

\node[font=\small, blue!70!black] at (2.65,.95) {$H_{\mathbf j}$};
\end{tikzpicture}
\caption{The described chain construction.}
\label{fig:chain-lower-bound}
\end{figure}
\section{Regime \texorpdfstring{$k$}{k}-fixed and \texorpdfstring{$n\rightarrow\infty$}{n to infinity} }
\subsection*{Upper bound proof}
For the sake of simplicity, let $k=2r \ge 6$, and denote our set of $n$ points as $X$. Assuming $X$ contains no empty convex $(2r+1)$-gons, we bound the number of empty convex $2r$-gons by double-counting pairs $(Q,K)$, where $K$ is an empty convex $2r$-gon and $Q$ is a specific type of sampled empty convex $r$-gon. For an empty convex $2r$-gon $K=(v_1,\dots,v_{2r})$ with indices modulo $2r$, we define two sampled $r$-gons by keeping one vertex and skipping one periodically. For  $t\in\{1,2\}$, let
$$Q^{(t)}(K):=(v_{t},v_{t+2},\dots,v_{t+2r-2}).$$
Each $2r$-gon $K$ admits exactly two such $r$-gons. Notice the structure this imposes on $K \setminus Q$: the $r$ skipped vertices must lie exactly one in each of the $r$ exterior wedges of the edges of $Q$. Here, by the exterior wedge associated with an edge, we mean the region spanned by this edge and the rays coming from the two neighbouring edges.

Now, consider an arbitrary empty convex $r$-gon $Q=(u_1, \dots, u_r)$. For each $i \in [r]$, let $e_i = u_i u_{i+1}$ be its edges (indices modulo $r$) and let $W_i$ be the corresponding open exterior wedge bounded  by $e_i$ and the supporting rays  $\overrightarrow{u_{i-1}u_i}$ and $\overrightarrow{u_{i+2}u_{i+1}}$ . Since \(Q\) is convex, these wedges are pairwise disjoint. A valid extension of $Q$ into a $2r$-hole is defined by an $r$-tuple of points $M = (m_1, \dots, m_r)$ where $m_i \in W_i$, such that inserting $m_i$ between $u_i$ and $u_{i+1}$ yields an empty convex $2r$-gon. We formalize a structural property of these extensions. 

\begin{lemma}\label{lemma:virtual_buffer}
 If two valid extensions $M = (m_1, \dots, m_r)$ and $N = (n_1, \dots, n_r)$ of $Q$ share a point such that $m_i = n_i$ for some $i$ and there are no $(2r+1)$ holes in the point set $X$, then $M = N$.
\end{lemma}

\begin{proof}
Suppose that \(M\neq N\). Since they agree somewhere, after a cyclic relabelling, we may assume that \(m_i=n_i\) but \(m_{i+1}\neq n_{i+1}\). After interchanging \(M\) and \(N\), assume that \(m_{i+1}\) lies on the exterior side of the edge \(u_{i+1}n_{i+1}\) of \(Q\cup N\). 
This makes the turn 
\(u_{i+1},m_{i+1},n_{i+1},u_{i+2}\) convex. 
If this chain was not convex, its only possible failure would be at \(n_{i+1}\), which would place \(n_{i+1}\) inside \(\triangle u_{i+1}m_{i+1}u_{i+2}\). This contradicts the emptiness of \(Q\cup M\). 
Among the points of \(X\) in \(\triangle u_{i+1}m_{i+1}n_{i+1}\), other than \(u_{i+1}\) and \(n_{i+1}\), choose \(z\) with minimum distance from the line \(u_{i+1}n_{i+1}\). Such a point exists because \(m_{i+1}\) is available. Then \(\triangle u_{i+1}zn_{i+1}\) is empty since every point in its interior is closer to this line.

The point \(z\) lies on the exterior side of the edge \(u_{i+1}n_{i+1}\), so the corner at \(z\) is convex. Since \(z\in\triangle u_{i+1}m_{i+1}n_{i+1}\), convexity of \(u_{i+1},m_{i+1},n_{i+1},u_{i+2}\)  shows that the corner at \(n_{i+1}\) also remains convex. Finally, convexity at \(u_{i+1}\) follows from \(m_i=n_i\) and the validity of both extensions. Thus, replacing \(u_{i+1}n_{i+1}\) by \(u_{i+1}z\) and
\(zn_{i+1}\) attaches exactly the empty triangle \(\triangle u_{i+1}zn_{i+1}\) to the empty polygon \(Q\cup N\), producing a \((2r+1)\)-hole, a contradiction.
\end{proof}
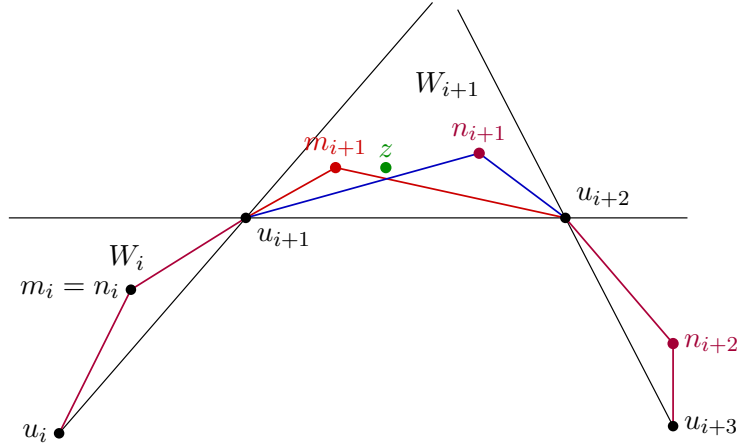
\begin{figure}[htbp]
\centering
\begin{tikzpicture}[
    scale=0.95,
    every node/.style={font=\small},
    boundary/.style={black, line width=0.45pt},
    aux/.style={black, line width=0.45pt},
    redpath/.style={red!80!black, line width=0.65pt},
    bluepath/.style={blue!75!black, line width=0.65pt},
    purplepath/.style={purple!90!black, line width=0.65pt},
    greenpath/.style={green!55!black, line width=0.55pt, dashed},
    pt/.style={circle, fill=black, inner sep=1.4pt},
    redpt/.style={circle, fill=red!80!black, inner sep=1.5pt},
    bluept/.style={circle, fill=blue!75!black, inner sep=1.5pt},
    purplept/.style={circle, fill=purple!85!black, inner sep=1.5pt},
    greenpt/.style={circle, fill=green!55!black, inner sep=1.5pt}
]

\coordinate (mui)  at (0,0);
\coordinate (mi)   at (1,2.0);
\coordinate (A)    at (2.60,3.00);
\coordinate (B)    at (7.05,3.00);
\coordinate (C)    at  (8.55,0.10);
\coordinate (m1)   at (3.85,3.70);
\coordinate (n1)   at (5.85,3.90);
\coordinate (z)    at (4.55,3.70);
\coordinate (n2)   at (8.55,1.25);

\draw[boundary] ($(mui)!-0!(A)$) -- ($(mui)!2!(A)$);
\draw[boundary] ($(C)!0!(B)$) -- ($(C)!2!(B)$);
\draw[aux] (-0.70,3.00) -- (8.75,3.00);

\draw[redpath]   (A) -- (m1) -- (B);
\draw[purplepath] (mui) -- (mi) -- (A)  ;
\draw[purplepath]   (B)-- (n2)-- (C);
\draw[bluepath]    (A) -- (n1) -- (B) ;

\node[pt] at (mui) {};
\node[pt] at (mi) {};
\node[pt] at (A) {};
\node[pt] at (B) {};
\node[pt] at (C) {};

\node[purplept] at (n1) {};
\node[redpt] at (m1) {};
\node[greenpt] at (z) {};
\node[purplept] at (n2) {};

\node[ left] at (mui) {$u_i$};
\node[ left] at (mi) {$m_i=n_i$};

\node[below right] at (A) {$u_{i+1}$};
\node[above right] at (B) {$u_{i+2}$};
\node[right] at (C) {$u_{i+3}$};

\node[right] at (4.80,4.85) {$W_{i+1}$};
\node[left] at (1.35,2.45) {$W_i$};

\node[above , purple!85!black] at (n1) {$n_{i+1}$};
\node[above, red!80!black] at (m1) {$m_{i+1}$};
\node[above, green!55!black] at (z) {$z$};
\node[right, purple!85!black] at (n2) {$n_{i+2}$};

\end{tikzpicture}
\caption{Local configuration in the proof of Lemma~\ref{lemma:virtual_buffer}.}
\label{fig:lemma31-visual-aid}
\end{figure}

We are now ready to finish the proof. 

\begin{proof}[Proof of Theorem~\ref{thm:upper-bound-fixed-k}]
\textbf{Even case}: Write \(k=2r\), and let \(X\) be an \(n\)-point set containing no
\((2r+1)\)-hole. Each \(2r\)-hole has exactly two alternating sampled
\(r\)-holes. Hence, if \(N(Q)\) denotes the number of valid extensions
of an empty \(r\)-hole \(Q\), then
$2M_{2r}(X)=\sum_Q N(Q).$ Fix \(Q\). By \Cref{lemma:virtual_buffer}, distinct extensions of \(Q\) cannot share a point. Each extension uses \(r\) points of \(X\setminus Q\),
and therefore $rN(Q)\le n-r.$ Since there are at most \(\binom nr\) empty \(r\)-holes, it follows that
$$
2M_{2r}(X)
\le
\binom nr\frac{n-r}{r}.
$$
Taking the maximum over all admissible \(X\), we obtain
$$
m_{2r,2r+1,n} \le\frac{n-r}{2r}\binom nr =\frac{2n-k/2}{k}\binom{n}{k/2}.
$$
\textbf{Odd case:} Now let \(k=2r+1\). Each \(k\)-hole has \(2r+1\) sampled \((r+1)\)-holes
$$
Q^{(t)}(K)=(v_t,v_{t+1},v_{t+3},\dots,v_{t+2r-1}),
$$
where we mark the unique unsplit edge \(e=v_tv_{t+1}\). Fix a marked sample \((Q,e)\), and place an auxiliary point \(p\) sufficiently close to the relative interior of \(e\) on its exterior side, so that appending \(p\) to any extension of \((Q,e)\) produces a \((2r+2)\)-hole in \(X\cup\{p\}\). The point set \(X\cup\{p\}\) contains no \((2r+3)\)-hole, since such a hole would contain a \((2r+2)\)-hole consisting entirely of points of \(X\). Thus, two distinct extensions of \((Q,e)\), after appending \(p\), would be two valid extensions of \(Q\) sharing the point \(p\), contradicting \Cref{lemma:virtual_buffer}. Hence every marked sample has at most one extension. Since there are at most \((r+1)\binom{n}{r+1}\) marked samples and every \((2r+1)\)-hole gives \(2r+1\) of them, double counting gives
$$
m_{2r+1,2r+2,n}
\le \frac{r+1}{2r+1}\binom{n}{r+1}=\frac{\lceil k/2\rceil}{k}\binom{n}{\lfloor k/2\rfloor+1}.
$$
\end{proof}
\begin{remark}
   Let us compare this bound for $n=\alpha k$ with $k\rightarrow\infty$ with the one obtained in \Cref{thm:three-halves}. Using  Stirling's formula, it turns out that for $\alpha\leq2.41$, the bound obtained in \Cref{thm:three-halves} provides a sharper estimate, but not for larger $\alpha$.
\end{remark}
\subsection*{Horton set  construction}

We use the standard recursive Horton sets. Let \(H=\{p_1,\dots,p_m\}\) be a finite
point set with increasing \(x\)-coordinates, and write \(H_0=\{p_2,p_4,\dots\}\) and
\(H_1=\{p_1,p_3,\dots\}\). A Horton set is defined recursively: sets of size at most one
are Horton sets, and for \(|H|>1\), both \(H_0\) and \(H_1\) are Horton sets, with one of them
lying high above the other. We shall use the standard facts that Horton sets exist in
every size, contain no \(7\)-hole, and are \(4\)-closed from above and from below; see
Horton~\cite{Horton} and Matou\v{s}ek~\cite[Section~3.2]{MatousekLectures}. Here, $4$-closed from above and from below means that every four-point cup or cap in the Horton set has a further point of the set on its
inward side, with $x$-coordinate between two consecutive vertices of
the cup or cap.
\begin{figure}[htbp]
    \centering
\begin{tikzpicture}[
    scale=0.85,
    every node/.style={font=\small},
    low/.style={circle,fill=blue!70!black,inner sep=1.7pt},
    high/.style={circle,fill=red!75!black,inner sep=1.7pt},
    closept/.style={circle,fill=orange!90!black,draw=black,inner sep=2pt},
    cup/.style={thick,blue!70!black},
    cap/.style={thick,green!45!black},
    hint/.style={dashed,gray!65}
]


\coordinate (a) at (0,0);
\coordinate (c) at (2,1.0);
\coordinate (e) at (4,0.45);
\coordinate (g) at (6,1.5);

\coordinate (b) at (1,4.2);
\coordinate (d) at (3,5.3);
\coordinate (f) at (5,4.9);
\coordinate (h) at (7,6.1);

\draw[hint] (b) -- (h);
\draw[hint] (a) -- (g);
\draw[hint] (d) -- (f);
\draw[hint] (c) -- (e);

\draw[blue!50] (a)--(c)--(e)--(g);
\draw[red!50]  (b)--(d)--(f)--(h);

\foreach \P in {a,c,e,g} \node[low] at (\P) {};
\foreach \P in {b,d,f,h} \node[high] at (\P) {};

\node[blue!70!black] at (6.7,1.1) {$H_0$};
\node[red!75!black] at (7.4,5.8) {$H_1$};

\begin{scope}[xshift=9.4cm]

\coordinate (p1) at (0,4.55);
\coordinate (p2) at (1.2,3.85);
\coordinate (p3) at (2.4,3.95);
\coordinate (p4) at (3.6,4.65);
\coordinate (q)  at (1.85,5.55);

\fill[orange!15] (p1)--(p2)--(p3)--(p4)--(3.6,5.9)--(0,5.9)--cycle;
\draw[cup] (p1)--(p2)--(p3)--(p4);
\draw[hint] (p1)--(p4);

\foreach \P/\lab in {p1/$p_1$,p2/$p_2$,p3/$p_3$,p4/$p_4$}
    \node[low,label=below:\lab] at (\P) {};

\node[closept] at (q) {};

\coordinate (r1) at (0,1.25);
\coordinate (r2) at (1.2,1.95);
\coordinate (r3) at (2.4,1.85);
\coordinate (r4) at (3.6,1.15);
\coordinate (s)  at (1.85,0.30);

\fill[orange!15] (r1)--(r2)--(r3)--(r4)--(3.6,-0.15)--(0,-0.15)--cycle;
\draw[cup] (r1)--(r2)--(r3)--(r4);
\draw[hint] (r1)--(r4);

\foreach \P/\lab in {r1/$p_1$,r2/$p_2$,r3/$p_3$,r4/$p_4$}
    \node[low,label=above:\lab] at (\P) {};

\node[closept] at (s) {};

\end{scope}

\end{tikzpicture}

    \caption{Construction of Horton set and illustration of the 4-closedness property}
    \label{fig:Horton}
\end{figure}
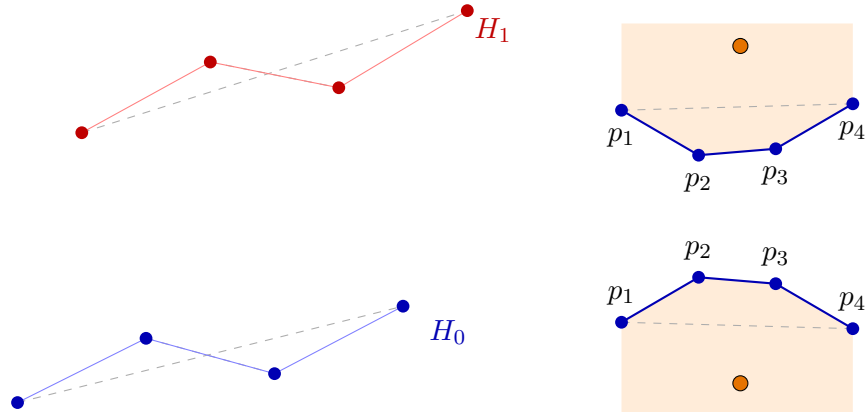

We shall also use the following flat perturbation form of the Horton construction.

\begin{lemma}\label{lem:horton-blocks}
Let \(t\ge 3\), and let \(v_1,\dots,v_t\) be the vertices of a strictly convex
\(t\)-gon in cyclic order. Let \(H_1^\ast,\dots,H_t^\ast\) be finite Horton sets.
Then there are affine copies \(H_i\) of \(H_i^\ast\), placed in pairwise disjoint
sufficiently small neighbourhoods of the vertices \(v_i\), such that the following hold.

\begin{enumerate}
\item If, for every \(i\in[t]\), the set \(C_i\subseteq H_i\) consists of three
consecutive points in the local \(x\)-order and forms a convex cup in the local
coordinates, then \(C_1\cup\cdots\cup C_t\) is a hole.

\item If \(Y\subseteq \bigcup_{i=1}^t H_i\) is a hole not contained in a single block,
then \(|Y\cap H_i|\le 3\) for every \(i\in[t]\).
\end{enumerate}
\end{lemma}

\begin{proof}
Choose pairwise disjoint neighbourhoods \(U_i\) of the vertices \(v_i\),
small enough to preserve their cyclic convex order. Inside each \(U_i\),
place a sufficiently small and flat affine copy \(H_i\) of \(H_i^\ast\),
oriented so that local cups face outwards.

Since only finitely many configurations are involved, the copies may be chosen so that the following two flatness properties hold. First, any three consecutive points forming a local cup appear as an
outward convex boundary chain, and the part of the resulting convex hull inside \(U_i\) contains no other point of \(H_i\). Second, if four points of \(H_i\) form a local cup or cap and \(z\in H_i\) closes it on its
inward side between two consecutive vertices \(q_j,q_{j+1}\), then \(z\in\operatorname{int}\triangle(q_j,q_{j+1},y)\) for every point
\(y\) from another block lying on that inward side.

For part~(1), let \(C_i\subseteq H_i\) be the chosen consecutive
three-point cups. By the first property, their union is in convex
position, with the points appearing block by block in cyclic order.
Moreover, no point of \(H_i\setminus C_i\) lies in its convex hull.
Applying this to every block shows that \(C_1\cup\cdots\cup C_t\) is a
hole.

For part~(2), let \(Y\) be a hole meeting more than one block, and
suppose that \(|Y\cap H_i|\ge4\). The points of \(Y\cap H_i\) form a
local cup or cap along the boundary of \(\operatorname{conv}(Y)\). Choose four consecutive such points \(q_1,q_2,q_3,q_4\). By the \(4\)-closedness of the Horton set, some \(z\in H_i\) closes this cup or
cap on its inward side, with its local \(x\)-coordinate between two consecutive points \(q_j,q_{j+1}\). In particular, \(z\notin Y\).
For any \(y\in Y\setminus H_i\), the second flatness property gives \(z\in\operatorname{int}\triangle(q_j,q_{j+1},y)\), and hence
\(z\in\operatorname{int}\operatorname{conv}(Y)\), contradicting that \(Y\) is a hole. Therefore \(|Y\cap H_i|\le3\) for every~\(i\).
\end{proof}
Now we are ready to finish the desired construction.
\begin{proof}[Proof of Theorem~\ref{thm:lower-bound-fixed-k}]
Write \(k=3r\) and \(n=mr+s\), where \(0\le s<r\). Let
\(m_i=m+1\) for \(i\le s\) and \(m_i=m\) otherwise, so that
\(\sum_{i=1}^r m_i=n\). For \(r\ge3\), place \(m_i\)-point Horton
blocks \(H_i\) around the vertices of a strictly convex \(r\)-gon as in
\Cref{lem:horton-blocks}. When \(r=2\), use the analogous placement of
two oppositely oriented flat blocks.

List the points of \(H_i\) in their local \(x\)-order. Of the \(m_i-2\)
triples of consecutive points, one starting-index parity consists of
cups and the other of caps. After reflecting the block if necessary, there is therefore a family \(\mathcal T_i\) of at least \((m_i-2)/2\) consecutive three-point cups. By \Cref{lem:horton-blocks}(1), choosing one triple from each
\(\mathcal T_i\) produces a \(3r=k\)-hole. Distinct choices produce
distinct holes, and hence
$$
M_k(X)\ge \prod_{i=1}^r|\mathcal T_i|
   \ge \prod_{i=1}^r\frac{m_i-2}{2}.$$
It remains to check that \(X\) contains no \((k+1)\)-hole. A hole meeting more than one block contains at most three points from each block by \Cref{lem:horton-blocks}(2), and therefore has at most
\(3r=k\) vertices. A hole contained in a single block has at most six vertices, since Horton sets contain no \(7\)-hole and consequently no larger hole. Thus \(X\) contains no \((k+1)\)-hole. Finally, \(m_i\ge\lfloor n/r\rfloor\ge n/r-1\), so
$$
m_{k,k+1,n}\ge M_k(X)
 \ge \left(\frac{n-3r}{2r}\right)^r
 =\left(\frac{3(n-k)}{2k}\right)^{k/3}.
$$
We finish  with a brief comment on how the case when $3\nmid k$ is handled. Let $p=k-3\lfloor\frac{k}{3}\rfloor$. We form $r=\lfloor \frac{k}{3}\rfloor$ Horton sets from $n-p$ points of size \(m_i\in\{\lfloor (n-p)/r\rfloor,\lceil (n-p)/r\rceil\}\). We arrange these blocks together with a block of the remaining $p$ points in the $\lceil \frac{k}{3}\rceil$-gon arranged as in the case above.
\end{proof}

\begin{figure}[htbp]
\centering
\begin{tikzpicture}[scale=.9]
\coordinate (b1) at (0.00,2.35);
\coordinate (b2) at (2.24,0.73);
\coordinate (b3) at (1.38,-1.90);
\coordinate (b4) at (-1.38,-1.90);
\coordinate (b5) at (-2.24,0.73);
\draw[guide] (b1)--(b2)--(b3)--(b4)--(b5)--cycle;

\begin{scope}[shift={(b1)}, rotate=0]
  \draw[black!18, fill=black!2] (0,0) ellipse (.72 and .20);
  \coordinate (h1a) at (-.55,-.05); \coordinate (h1b) at (-.33,-.08); \coordinate (h1c) at (-.11,.13); \coordinate (h1d) at (.11,-.07); \coordinate (h1e) at (.33,.02); \coordinate (h1f) at (.55,-.04);
  \draw[black!55] (h1a)--(h1b)--(h1c)--(h1d)--(h1e)--(h1f);
  \foreach \p in {h1a,h1e,h1f}{\node[blockpt] at (\p) {};}
  \foreach \p in {h1b,h1c,h1d}{\node[blockchosen] at (\p) {};}
  \draw[hole] (h1b)--(h1c)--(h1d);
\end{scope}
\begin{scope}[shift={(b2)}, rotate=-72]
  \draw[black!18, fill=black!2] (0,0) ellipse (.72 and .20);
  \coordinate (h2a) at (-.55,-.05); \coordinate (h2b) at (-.33,-.08); \coordinate (h2c) at (-.11,.13); \coordinate (h2d) at (.11,-.07); \coordinate (h2e) at (.33,.02); \coordinate (h2f) at (.55,-.04);
  \draw[black!55] (h2a)--(h2b)--(h2c)--(h2d)--(h2e)--(h2f);
  \foreach \p in {h2a,h2e,h2f}{\node[blockpt] at (\p) {};}
  \foreach \p in {h2b,h2c,h2d}{\node[blockchosen] at (\p) {};}
  \draw[hole] (h2b)--(h2c)--(h2d);
\end{scope}
\begin{scope}[shift={(b3)}, rotate=-144]
  \draw[black!18, fill=black!2] (0,0) ellipse (.72 and .20);
  \coordinate (h3a) at (-.55,-.05); \coordinate (h3b) at (-.33,-.08); \coordinate (h3c) at (-.11,.13); \coordinate (h3d) at (.11,-.07); \coordinate (h3e) at (.33,.02); \coordinate (h3f) at (.55,-.04);
  \draw[black!55] (h3a)--(h3b)--(h3c)--(h3d)--(h3e)--(h3f);
  \foreach \p in {h3a,h3e,h3f}{\node[blockpt] at (\p) {};}
  \foreach \p in {h3b,h3c,h3d}{\node[blockchosen] at (\p) {};}
  \draw[hole] (h3b)--(h3c)--(h3d);
\end{scope}
\begin{scope}[shift={(b4)}, rotate=-216]
  \draw[black!18, fill=black!2] (0,0) ellipse (.72 and .20);
  \coordinate (h4a) at (-.55,-.05); \coordinate (h4b) at (-.33,-.08); \coordinate (h4c) at (-.11,.13); \coordinate (h4d) at (.11,-.07); \coordinate (h4e) at (.33,.02); \coordinate (h4f) at (.55,-.04);
  \draw[black!55] (h4a)--(h4b)--(h4c)--(h4d)--(h4e)--(h4f);
  \foreach \p in {h4a,h4e,h4f}{\node[blockpt] at (\p) {};}
  \foreach \p in {h4b,h4c,h4d}{\node[blockchosen] at (\p) {};}
  \draw[hole] (h4b)--(h4c)--(h4d);
\end{scope}
\begin{scope}[shift={(b5)}, rotate=-288]
  \draw[black!18, fill=black!2] (0,0) ellipse (.72 and .20);
  \coordinate (h5a) at (-.55,-.05); \coordinate (h5b) at (-.33,-.08); \coordinate (h5c) at (-.11,.13); \coordinate (h5d) at (.11,-.07); \coordinate (h5e) at (.33,.02); \coordinate (h5f) at (.55,-.04);
  \draw[black!55] (h5a)--(h5b)--(h5c)--(h5d)--(h5e)--(h5f);
  \foreach \p in {h5a,h5e,h5f}{\node[blockpt] at (\p) {};}
  \foreach \p in {h5b,h5c,h5d}{\node[blockchosen] at (\p) {};}
  \draw[hole] (h5b)--(h5c)--(h5d);
\end{scope}

\draw[hole, opacity=.38] (h1b)--(h1c)--(h1d)--(h2b)--(h2c)--(h2d)--(h3b)--(h3c)--(h3d)--(h4b)--(h4c)--(h4d)--(h5b)--(h5c)--(h5d)--cycle;

\end{tikzpicture}
\caption{Construction for $r=5$ }
\label{fig:horton-block-lower-bound}
\end{figure}
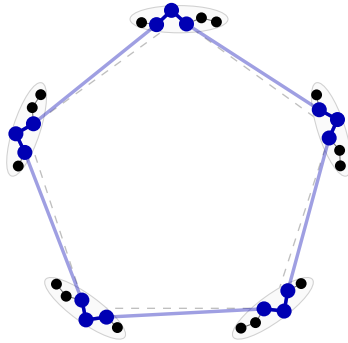

\section*{Declaration of AI use}
The authors acknowledge the help of OpenAI's ChatGPT with TikZ.

\section*{Acknowledgements}
We would like to thank V\'it Jel\'inek and Pavel Valtr for many helpful conversations
during the Seminar on Combinatorial Problems at Charles University, where most of the
work on the paper was done. This work was supported by grant GA\v{C}R 23-04949X.
A. D\v{z}uklevski was also supported by GAUK grant VV--2025--2608/22 and by the
Visegrad Fund.

\bibliographystyle{cas-model2-names}
\bibliography{Holes/Bibliography}

@article{ErdosSzekeres1935,
  author  = {Paul Erd\H{o}s and George Szekeres},
  title   = {A Combinatorial Problem in Geometry},
  journal = {Compositio Mathematica},
  volume  = {2},
  pages   = {463--470},
  year    = {1935}
}

@article{Erdos1978,
  author  = {Paul Erd\H{o}s},
  title   = {Some More Problems on Elementary Geometry},
  journal = {Australian Mathematical Society Gazette},
  volume  = {5},
  number  = {2},
  pages   = {52--54},
  year    = {1978}
}

@article{Harborth1978,
  author  = {Heiko Harborth},
  title   = {Konvexe F\"unfecke in ebenen Punktmengen},
  journal = {Elemente der Mathematik},
  volume  = {33},
  pages   = {116--118},
  year    = {1978}
}

@article{Nicolas2007,
  author  = {Carlos M. Nicol\'as},
  title   = {The Empty Hexagon Theorem},
  journal = {Discrete \& Computational Geometry},
  volume  = {38},
  number  = {2},
  pages   = {389--397},
  year    = {2007},
  doi     = {10.1007/s00454-007-1343-6}
}

@article{Gerken2008,
  author  = {Tobias Gerken},
  title   = {Empty Convex Hexagons in Planar Point Sets},
  journal = {Discrete \& Computational Geometry},
  volume  = {39},
  number  = {1--3},
  pages   = {239--272},
  year    = {2008},
  doi     = {10.1007/s00454-007-9018-x}
}

@article{Overmars2002,
  author  = {Mark Overmars},
  title   = {Finding Sets of Points without Empty Convex 6-Gons},
  journal = {Discrete \& Computational Geometry},
  volume  = {29},
  number  = {1},
  pages   = {153--158},
  year    = {2002}
}

@incollection{HeuleScheucher2024,
  author    = {Marijn J. H. Heule and Manfred Scheucher},
  title     = {Happy Ending: An Empty Hexagon in Every Set of 30 Points},
  booktitle = {Tools and Algorithms for the Construction and Analysis of Systems},
  series    = {Lecture Notes in Computer Science},
  volume    = {14570},
  pages     = {61--80},
  publisher = {Springer, Cham},
  year      = {2024},
  doi       = {10.1007/978-3-031-57246-3_5}
}

@article{KatchalskiMeir1988,
  author  = {Meir Katchalski and Amram Meir},
  title   = {On Empty Triangles Determined by Points in the Plane},
  journal = {Acta Mathematica Hungarica},
  volume  = {51},
  number  = {3--4},
  pages   = {323--328},
  year    = {1988}
}

@article{BaranyFuredi1987,
  author  = {Imre B\'ar\'any and Zolt\'an F\"uredi},
  title   = {Empty Simplices in Euclidean Space},
  journal = {Canadian Mathematical Bulletin},
  volume  = {30},
  number  = {4},
  pages   = {436--445},
  year    = {1987}
}

@phdthesis{Dehnhardt1987,
  author = {Klemens Dehnhardt},
  title  = {Leere konvexe Vielecke in ebenen Punktmengen},
  school = {Universit\"at Bonn},
  year   = {1987}
}

@article{AichholzerEtAl2020,
  author  = {Oswin Aichholzer and Martin Balko and Thomas Hackl and Jan Kyn\v{c}l and Irene Parada and Manfred Scheucher and Pavel Valtr and Birgit Vogtenhuber},
  title   = {A Superlinear Lower Bound on the Number of 5-Holes},
  journal = {Journal of Combinatorial Theory, Series A},
  volume  = {173},
  pages   = {105236},
  year    = {2020}
}

@misc{AstudilloMarbanSolePi2026,
  author       = {Omar Astudillo-Marb\'an and Oriol Sol\'e-Pi},
  title        = {There Are Many 5-Holes},
  year         = {2026},
  eprint       = {2603.18484},
  archivePrefix= {arXiv},
  primaryClass = {math.CO}
}

@article{Horton,
  author  = {Joseph D. Horton},
  title   = {Sets with No Empty Convex 7-Gons},
  journal = {Canadian Mathematical Bulletin},
  volume  = {26},
  number  = {4},
  pages   = {482--484},
  year    = {1983},
  doi     = {10.4153/CMB-1983-077-8}
}

@book{MatousekLectures,
  author    = {Ji\v{r}{\'\i} Matou\v{s}ek},
  title     = {Lectures on Discrete Geometry},
  publisher = {Springer},
  year      = {2002},
  doi       = {10.1007/978-1-4613-0039-7}
}

@article{BaranyValtr2004,
  title={Planar point sets with a small number of empty convex polygons},
  author={B{\'a}r{\'a}ny, Imre and Valtr, Pavel},
  journal={Studia Scientiarum Mathematicarum Hungarica},
  volume={41},
  number={2},
  pages={243--266},
  year={2004},
  publisher={Akad{\'e}miai Kiad{\'o}}
}

@misc{SukZhou2026,
      title={On the maximum number of $k$-holes in point sets with no $(k + 1)$-hole}, 
      author={Andrew Suk and Su Zhou},
      year={2026},
      eprint={2606.05721},
      archivePrefix={arXiv},
      primaryClass={math.CO},
      url={https://arxiv.org/abs/2606.05721}, 
}

\end{document}